\documentclass[11pt]{amsart}

\usepackage{amssymb}

\usepackage{hyperref} 

\theoremstyle{definition}

\theoremstyle{remark}

\numberwithin{equation}{section}

\newtheoremstyle{fancy}{}{}{\itshape}{}{\textsc\bgroup}{.\egroup}{ }{}
\newtheoremstyle{fanci}{}{}{\rm}{}{\textsc\bgroup}{.\egroup}{ }{}

\theoremstyle{fancy}
\newcounter{intro}

\numberwithin{equation}{section} \swapnumbers
\newtheorem{cor}[equation]{Corollary}
\newtheorem{lem}[equation]{Lemma}
\newtheorem{prop}[equation]{Proposition}
\newtheorem{thm}[equation]{Theorem}

\newtheorem*{result}{Result}

\theoremstyle{fanci}

\newtheorem{dfn}[equation]{Definition}

\newtheorem{rem}[equation]{Remark}

\newcommand{\mg}{\mathfrak{g}}

\newcommand{\mt}{\mathfrak{t}}

\newcommand{\R}{\mathbb{R}}

\newcommand{\MM}{\mathcal{M}}
\newcommand{\mU}{\mathcal{U}}

\newcommand{\germ}{\mathfrak}

\newcommand{\trace}{\operatorname{Tr}}

\newcommand{\ad}{\operatorname{ad}}
\newcommand{\Ad}{\text{Ad}}
\newcommand{\SO}{\operatorname{SO}}

\newcommand{\Spec}{\operatorname{Spec}}

\newcommand{\vol}{\operatorname{vol}}

\newcommand{\SL}{\operatorname{SL}}
\newcommand{\Id}{\operatorname{Id}}
\newcommand{\MIG}{\MM_{\it{left}}(G)}

\newcommand{\Mat}{\operatorname{Mat}}
\renewcommand{\mU}{\,\mathcal{U}}
\newcommand{\mL}{\mathcal{L}}
\newcommand{\restr}{\upharpoonright}

\begin{document}

\newcommand{\spacing}[1]{\renewcommand{\baselinestretch}{#1}\large\normalsize}
\spacing{1.14}

\title[Spectral Isolation of Bi-Invariant Metrics]{Spectral Isolation of
Bi-invariant Metrics on Compact Lie Groups}


\author[C.S. Gordon]{Carolyn S. Gordon $^\dagger$}
\address{Dartmouth College\\ Department of Mathematics \\ Hanover, NH 03755}
\email{carolyn.s.gordon@dartmouth.edu}
\thanks{$^\dagger$Research partially supported by National Science Foundation
grant DMS 0605247}


\author[D. Schueth]{Dorothee Schueth $^\ddagger$}
\address{Humboldt-Universit\"{a}t \\ Institut f\"{u}r Mathematik \\ Berlin,
Germany}
\email{schueth@math.hu-berlin.de}
\thanks{$^\ddagger$Research partially supported by DFG
Sonderforschungsbereich~647. This work was started during a 
visit of the second author to Dartmouth College,
which she would like to thank for its great hospitality.}

\author[C. J. Sutton]{Craig J. Sutton $^\sharp$}
\address{Dartmouth College\\ Department of Mathematics \\ Hanover, NH 03755}
\email{craig.j.sutton@dartmouth.edu}
\thanks{$^\sharp$Research partially supported by National Science Foundation
grant DMS 0605247}

\subjclass[2000]{53C20, 58J50}
\keywords{Laplacian, eigenvalue spectrum, Lie group, left-invariant metric,
bi-invariant
metric}

\begin{abstract}
We show that a bi-invariant metric on a compact connected Lie group $G$ is
spectrally isolated within the class of left-invariant metrics. In fact, we
prove that given a bi-invariant metric $g_0$ on $G$ there is a positive integer
$N$ such that, within a neighborhood of $g_0$ in the class of left-invariant
metrics of at most the same volume, $g_0$ is uniquely determined by the first
$N$ distinct non-zero eigenvalues of its Laplacian (ignoring multiplicities). In
the case where $G$ is simple, $N$ can be chosen to be two.

\medskip
\noindent
R\'ESUM\'E. Soit $G$ un groupe de Lie compact et connexe, et soit $g_0$ une
m\'etrique bi-invariante sur $G$.
On d\'emontre que $g_0$ est isol\'ee spectralement
dans la classe des m\'etriques invariantes \`a gauche:
Plus pr\'ecis\'ement, il existe un entier positif $N$ tel que, dans un voisinage
de $g_0$ dans la
classe des m\'etriques invariantes \`a gauche et de volume \'egal ou inf\'erieur
\`a
celui de $g_0$, la m\'etrique
$g_0$ est determin\'ee de mani\`ere unique par les $N$ premi\`eres valeurs
propres strictement
positives de son Laplacien (sans multiplicit\'es). Si $G$ est simple, on peut
choisir $N=2$.

\end{abstract}

\maketitle

\section{Introduction}

\noindent
Given a connected closed Riemannian manifold $(M,g)$ its {\bf spectrum}, denoted
$\Spec(M,g)$, is defined to be the sequence of eigenvalues, counted with
multiplicities, of the associated Laplacian $\Delta$ acting on smooth functions. 
Two Riemannian manifolds $(M_{1}, g_{1})$ and $(M_{2}, g_{2})$ are said to 
be {\bf isospectral} if their spectra (counting multiplicities) agree. 
Inverse spectral geometry is the study of the extent to which 
geometric properties of a Riemannian manifold $(M,g)$ are determined by its
spectrum.

A long standing question is whether very special Riemannian manifolds -- e.g., 
manifolds of constant curvature or symmetric spaces -- may be spectrally
distinguishable from other Riemannian manifolds.  The strongest results are for
constant curvature: Tanno  showed that a round sphere of dimension at most six 
is uniquely determined by its spectrum among all \emph{orientable} 
Riemannian manifolds \cite{Tanno}, and in arbitrary dimensions round metrics on
spheres 
are at least spectrally isolated among all Riemannian metrics on spheres
\cite{Tanno2}.  
In contrast, the first and third author have shown that in
dimension $7$ and higher
there are isospectrally deformable metrics on spheres arbitrarily close to the 
standard metric 
\cite{Gordon, Schueth2}.
Hence, for geometries that are in some sense extremely close to being ``nice'' 
or ``ideal'', spectral uniqueness can fail profoundly.  

While many examples
exist of isospectral flat manifolds, Kuwabara \cite{Ku} has proven that flat
metrics are at least spectrally
isolated within the space of all metrics.  However, even the question of whether
a flat torus may be isospectral to a non-flat manifold remains open!  One cannot
resolve
this question by appealing to the heat invariants of a Riemannian manifold as
there are
examples of non-flat manifolds all of whose heat invariants vanish
\cite{Patodi}. 

Outside of the setting of constant curvature, we are not aware of any examples
of Riemannian metrics that are known to be spectrally isolated among arbitrary
Riemannian metrics.  Various results show that \emph{within 
certain classes} of Riemannian metrics, isospectral sets are finite.   Even
here, many of the results involve constant curvature.
For example,  isospectral sets of flat tori are
finite (see \cite{Wolpert} or unpublished work of Kneser) as are isospectral
sets of Riemann surfaces \cite{McKean}. 
As for the class of symmetric spaces, the first and third author have recently
shown that any collection 
of mutually isospectral compact symmetric spaces is finite \cite{GorSut}. 

This article is motivated by the question of whether one can tell from the
spectrum whether a compact Riemannian manifold is symmetric.  Given that this
question has resisted solution even in the case of spheres,  it does not appear
tractable at this time to compare the spectrum of a symmetric space with that of
a completely arbitrary Riemannian  manifold. 
Instead, we ask whether symmetric spaces can be spectrally distinguished within 
a larger class of homogeneous Riemannian manifolds.  

The compact symmetric spaces fall into two types; the type we consider are those
given by bi-invariant Riemannian metrics on compact (not necessarily semisimple) 
Lie groups.   We compare the spectrum of each such symmetric space with the
spectra of arbitrary left-invariant metrics on the Lie group.  As a departure
point we note that the second author showed that there are no non-trivial
continuous
isospectral deformations of a bi-invariant metric within the class of
left-invariant metrics on a compact Lie group \cite{Schueth}.  
This prompts one to ask whether a bi-invariant metric on a compact Lie group 
$G$ is spectrally isolated within the class of left-invariant metrics.  
We give an affirmative answer; in fact we obtain a significantly stronger
result.

Let $\MIG$ denote the set of left-invariant metrics on a Lie group~$G$.
This set can be canonically identified with the set of Euclidean
inner products on the Lie algebra of~$G$. The latter set can in turn be
identified, after some choice of basis, with the set of positive definite
symmetric
$(m\times m)$-matrices, where $m$ is the dimension of~$G$. The
canonical topology on this set of matrices gives rise to a topology on $\MIG$
(independent of the choice of basis), and it is this topology that we consider.
We call a left-invariant metric $g_0$ on~$G$ {\bf spectrally isolated} in $\MIG$
if it is locally spectrally determined within $\MIG$; that is, there is
a neighborhood~$\mU$ of~$g_0$ in $\MIG$ such that no $g\in\mU\setminus\{g_0\}$
is isospectral to~$g_0$. We prove the following:

\begin{result} Let $g_0$ be a bi-invariant metric on a compact Lie group $G$.
\begin{enumerate}
\item There is a neighborhood $\mU$ of $g_0$ in $\MIG$ and a positive integer
$N$ such that if $g$ is any metric in $\mU$ with $\vol(g) \leq \vol(g_0)$ and
whose first $N$
distinct eigenvalues {\rm(}ignoring multiplicities{\rm)} agree with those of
$g_0$, then $g$ is isometric to $g_0$. (See Theorem~\ref{Thm:Main}.)

\item The metric $g_0$ is spectrally isolated in $\MIG$. (See
Corollary~\ref{Cor:Main}.)

\item Let $\alpha_1<\alpha_2<\alpha_3$ be three distinct consecutive eigenvalues
{\rm(}ignoring multiplicities{\rm)} of the associated Laplacian $\Delta_0$.
If $G$ is \emph{simple}, then there exists a neighborhood $\mU$ of~$g_0$ in
$\MIG$ such that if $g \in \mU$
satisfies $\vol(g) \leq \vol(g_0)$ and the condition that three consecutive
distinct eigenvalues of~$\Delta_g$ agree with $\alpha_1$, $ \alpha_2$ and $
\alpha_3$ {\rm(}again ignoring multiplicities{\rm)}, then $g = g_0$.  In
particular,
letting $\alpha_1=0$, then the first two distinct non-zero eigenvalues along
with the
volume bound distinguish $g_0$ withing $\mU$. (See
Theorem~\ref{Thm:SimpleLocRigidity}.)

\end{enumerate}
\end{result}

\noindent
The second result above is immediate from the first since the spectrum of a
compact Riemannian manifold determines its volume.  Hence, within the class of
left-invariant metrics on a compact Lie group $G$, any metric $g\ne g_0$ that is 
isospectral to a bi-invariant metric $g_0$ must be sufficiently far away from
$g_0$.  In contrast, we note that the second author exhibited the first examples
of continuous isospectral families of left-invariant metrics on compact simple
Lie groups \cite{Schueth}; see also \cite{Proctor}.

In light of the fact that most examples of isospectral manifolds in the
literature exploit metrics with ``large'' symmetry groups, the spectral
isolation results above lend strong support to the conjecture that a
bi-invariant metric on a compact Lie group $G$ is spectrally isolated within the
class of \emph{all} metrics on $G$. In fact, these results lead one to speculate
on whether a bi-invariant metric on a \emph{semisimple} Lie group is uniquely
determined by its spectrum.\footnote{We must restrict our attention to
semisimple Lie groups due to the existence of nontrivial pairs of isospectral
flat tori (e.g., \cite{Milnor} and \cite{CS}).} In Section~\ref{Sec:Simple} we
present strong evidence that the bi-invariant metric on a
compact simple Lie group is uniquely determined by its spectrum within the class
of left-invaraint metrics. In particular, we show the following.
 
\begin{result}
Let $g_0$ be a bi-invariant metric on a compact simple Lie group $G$, and let $g
\ne g_0$ be a left-invariant metric on $G$, which is isospectral to $g_0$. Then
there is a constant $C \equiv C(g) >1 $, such that for every subspace $V \leq
L^2(G)$ that is invariant under the right regular action of $G$, we have 
$$\frac{\trace(\Delta_g \upharpoonright V)}{\trace(\Delta_0 \upharpoonright V)}
\equiv C > 1.$$
(See Proposition~\ref{prop:simple} for a more precise statement.)
\end{result}
\noindent
This implies that  if $g \neq g_0 \in \MIG$ is isospectral to the bi-invariant
metric $g_0$, then a very special rearrangement of the eigenvalues must occur. 

The outline of this paper is as follows. In Section~\ref{Sec:Main} we establish
the main results for bi-invariant metrics on arbitrary compact Lie groups. In
Section~\ref{Sec:Simple} we restrict our attention to compact simple Lie groups
to obtain the stronger results in this setting.


\section{Proof of the main result}\label{Sec:Main}

\noindent
Following \cite{LMNR} we introduce the notion of {\bf eigenvalue equivalence},
which is weaker than that of isospectrality. 
The same notion was introduced earlier by Z.I.~Szabo \cite[p.~212]{Szabo}, who
referred
to it as isotonality.  We also define a notion of partial eigenvalue
equivalence.

\begin{dfn}\label{def.eig}
Given  a compact Riemannian manifold $(M,g)$, we define the {\bf eigenvalue set}
of $(M,g)$ to be the ordered collection
of eigenvalues of the associated Laplace operator $\Delta_g$ on functions
on~$M$,
\emph{not} counting multiplicities. We will say that two compact Riemannian
manifolds
are
{\bf eigenvalue equivalent} if their eigenvalue sets coincide.    For $N$ a
positive integer, we will say that two compact Riemannian manifolds
 are {\bf eigenvalue equivalent up to level $N$} if the first $N$ elements of
their eigenvalue sets coincide.
\end{dfn}

\begin{lem}\label{Lem:TraceNbhd}
Let $G$ be a compact Lie group, and let $g_0$ be a bi-invariant metric on~$G$
with associated
Laplacian~$\Delta_0$. Let $V \leq L^2(G)$ be a finite dimensional subspace
which is invariant under the right-regular representation of~$G$ on
$L^2(G)$. Then there exists a positive integer $N$ and a 
neighborhood $\mU$ of $g_0$ in $\MIG$ such that if $g \in \mU$ is eigenvalue
equivalent to $g_0$ up to level $N$, then
$$\Delta_g \upharpoonright V = \Delta_0 \upharpoonright V.$$
\end{lem}

\begin{proof}
First note that $V$, being a finite dimensional subspace of $L^2(G)$ which is
invariant under
the right-regular representation, contains only smooth functions.
Moreover, $V$ is a direct sum of finitely many irreducible
representations of $G$; it is therefore enough to prove the result
in the case that $V$ is irreducible.
For any $g \in \MIG$ and any $g$-orthonormal basis $\{ Y_1, \ldots , Y_n \}$
of the Lie algebra of~$G$,
the associated Laplace operator on smooth functions on~$G$ is given by
$$\Delta_g = - \sum_{j=1}^{n} (\rho_{*}Y_j)^2,$$
where $\rho : G \to U(L^{2}(G))$ is the right-regular representation of~$G$.
Thus, $V$ is invariant under~$\Delta_g$.
Since $g_0$ is bi-invariant, right translations in~$G$ are $g_0$-isometries;
hence $\Delta_0:V\to V$ commutes
with the action of $G$ on~$V$. Irreducibility of~$V$ implies by Schur's Lemma
that $\Delta_0\upharpoonright V$ is a multiple of the identity, say
$\Delta_0\upharpoonright V = \lambda\Id_V$.
We may choose $\epsilon > 0$ such that $(\lambda - \epsilon, \lambda + \epsilon)
\cap
\Spec(\Delta_0) = \{ \lambda \}$.  Choose $N$ large enough so that the $N$th
element of the eigenvalue set is greater than $\lambda$ (and hence greater than
$\lambda+\epsilon$).
The hermitian operators $\Delta_g\upharpoonright V$ on the finite
dimensional vector space~$V$ depend continuously on~$g$. Therefore, their
eigenvalues also depend continuously on~$g$.
Consequently, there is a neighborhood $\mU$ of~$g_0$ in $\MIG$ such that for
each
$g\in\mU$ the eigenvalues of $\Delta_g\upharpoonright V$ must lie
in $(\lambda - \epsilon, \lambda + \epsilon)$.    If
$g \in \mU$ is eigenvalue equivalent to $g_0$ up to level $N$, it follows that
$\Delta \upharpoonright V= \Delta_0 \upharpoonright V = \lambda
\Id_V$.
\end{proof}

We now establish the spectral isolation of bi-invariant metrics
on compact connected Lie groups. In fact, we prove a little more; namely,
we replace the isospectrality condition by the much weaker condition of
partial eigenvalue equivalence together with an upper volume bound.

\begin{thm}\label{Thm:Main}
Let $g_0$ be a bi-invariant metric on a compact connected Lie group~$G$.
Then there is a positive integer $N$, depending only on $g_0$, and a
neighborhood $\mU$ of $g_0$ in $\MIG$ such that if $g \in \mU$
is eigenvalue
equivalent to $g_0$ up to order $N$ and satisfies $\vol(g) \leq \vol(g_0)$, then
$g = g_0$.
\end{thm}

\begin{cor}\label{Cor:Main}
Let $g_0$ be a bi-invariant metric on a compact connected Lie group~$G$.
Then $g_0$ is spectrally isolated in $\MIG$.
\end{cor}

The corollary follows from the theorem by the fact that isospectrality implies
eigenvalue equivalence and
equality of volumes; in fact, the volume is the first of the classical heat
invariants.

\begin{proof}[Proof of Theorem~\ref{Thm:Main}]
We have $G=G_{ss}T$ where $G_{ss}$ is semisimple, $T$ is a torus and $G_{ss}\cap
T$ is
finite.  The Lie algebra of~$G$ is a direct sum $\mg=\mg_{ss}\oplus \mt$, where
$\mg_{ss}$ and~$\mt$ are the Lie algebras of $G_{ss}$ and~$T$. In particular,
$G_{ss}$ and~$T$ commute, and $\mg_{ss}$ is $g_0$-orthogonal to~$\mt$.

We first claim that there exists a positive integer $N'$ and a
neighborhood~$\mU'$ of~$g_0$ in $\MIG$ such that if
$g\in\mU'$ is
eigenvalue equivalent to~$g_0$ up to level $N'$, then $g$ and~$g_0$, viewed as
inner products on~$\mg$,
induce the same inner product on $\mg/\mg_{ss}$.
By the inner product induced by $g$ we mean the one obtained by identifying
$\mg/\mg_{ss}$
with the $g$-orthogonal complement of $\mg_{ss}$ in~$\mg$.
To prove the claim,
note that the Lie group $\overline{T}:=G/G_{ss}\cong T/(G_{ss}\cap T)$
is a torus which is finitely covered by~$T$.  In particular, the Lie algebra
of~$\overline{T}$ is canonically identified with~$\mt$.  Let $p:G\to
\overline{T}$ be the homomorphic projection.  Given $g\in\MIG$, denote by~$\bar
g$
the induced (flat) metric on~$\overline{T}$ (i.e., the metric
for which $p:(G,g)\to(\overline{T},\bar g)$ becomes a Riemannian submersion).
Let $\mL$ be the lattice in~$\mt$ which is the kernel
of the Lie group exponential map $\mt\to\overline{T}$, and let
$\mL^*\subset\mt^*$
be the dual lattice. For $\mu\in\mL^*$, denote by $\|\mu\|_{\bar g}$ the
norm of~$\mu$ with respect to the dual inner product on~$\mt^*$.
Let $\nu_1,\ldots,\nu_k$ be a basis of~$\mL^*$, where $k=\dim(T)$.
Write $L:=k+\binom k2$, and let $\{\mu_1,\ldots,\mu_L\}$ be the set
containing the vectors $\nu_i$ as well as the $\nu_i+\nu_j$ for $i\ne j$.  Note
that, by polarization, the norm $\|\,.\,\|_{\bar g}$ on~$\mt^*$ -- and
hence~$\bar g$
itself -- is uniquely determined by the
norms of the vectors $\mu_1,\ldots,\mu_L$.
For each $s\in\{1,\ldots,L\}$ let $\bar f_s:\overline{T}\to U(1)$ (where
$U(1)$ is the unitary group of unit complex numbers) be the associated character
of~$\overline{T}$. Then
$\Delta_{\bar g}\bar f_s=4\pi^2\|\mu_s\|_{\bar g}^2
\bar f_s$.  Now $f_s:=\bar f_s\circ p$ is a character
on~$G$. Since the Riemannian submersion $p:G\to\overline{T}$ has minimal
fibers, $f_s$~is an eigenfunction of~$\Delta_g$ with eigenvalue
$4\pi^2\|\mu_s\|_{\bar g}^2$ for each $s=1,\dots, L$. (One can also
verify this fact by direct computation.)

The one-dimensional space
$F_s\leq L^2(G)$ spanned by the character~$f_s$ is invariant under the
right-regular representation.
Let $N'$ be a positive integer and $\mU'$ be a neighborhood of~$g_0$
in~$\MIG$ satisfying the property from Lemma~\ref{Lem:TraceNbhd} with
respect to $F_1\oplus\ldots\oplus F_L$, and let
$g\in\mU'$ be eigenvalue equivalent to~$g_0$ up to level $N'$. Then we must have
$\|\mu_s\|_{\bar g}=\|\mu_s\|_{\bar g_0}$ for each $s=1,\dots, L$.
As remarked above, this implies $\bar g=\bar g_0$. The claim follows.

In the case of the bi-invariant metric $g_0$, the metric $\bar g_0$ on
$\mt$ coincides with the restriction of the inner product $g_0$ to $\mt$ since
$\mg_{ss}$ and~$\mt$ are $g_0$-orthogonal.
However, for more general $g$, one has only that the differential
$p_*:\mg\to\mt$
of the
projection $p:G\to\overline{T}$ restricts to an inner product space isometry
$p_*:(\mg_{ss}^{\perp_g}, g)\to (\mt,\bar{g})$.  In particular, if $g\in
\mU'$, then it follows from the claim that
\begin{equation}\label{eq.isom}
\mbox{the projection from $(\mg_{ss}^{\perp_g},g)$ to 
$(\mt,g_0)$ along $\mg_{ss}$ is an isometry.}
\end{equation}

For the remaining part of the argument, let
$\mg_{ss}=\mg_1\oplus\ldots\oplus\mg_r$ be the decomposition
of~$\mg_{ss}$ into simple Lie subalgebras.
The adjoint representation of~$G$ on~$\germ{g}$,
restricted to the invariant subspace $\germ{g}_\ell$, is an irreducible
representation
of~$G$ for each $\ell=1,\ldots,r$. Note that for $\ell\ne\ell'$ these
representations
of~$G$ are inequivalent even in the case when $\germ{g}_\ell$ and
$\germ{g}_{\ell'}$
happen to be isomorphic as Lie algebras.
By the Peter-Weyl Theorem, every irreducible representation of~$G$ occurs in the
right-regular representation of~$G$ on $L^2(G)$ (with multiplicity equal to its
dimension).
Thus, for each $\ell=1,\ldots, r$ we can choose
a corresponding irreducible subspace $V_\ell\leq L^2(G)$, and the action of $G$
on the subspace $V_1\oplus\ldots\oplus V_r$ of $L^2(G)$ will then be equivalent
to the adjoint representation
of~$G$
acting on~$\mg_{ss}$. 

Let $N''$ be a positive integer and $\mU''$ be a neighborhood of~$g_0$ in $\MIG$
satisfying the property from Lemma~\ref{Lem:TraceNbhd} with respect to
$V_1\oplus\ldots\oplus V_r$. We are going to show that $N:=
\text{max}\{N',N''\}$ and $\mU:=\mU'\cap\mU''$
satisfy the property stated in the Theorem.

If $g$ is any left-invariant metric  on $G$ and $\{U_1,\dots, U_m\}$ is a
$g$-orthonormal basis of $\mg$, then
\begin{equation}\label{eq.tr}
 \trace(\Delta_g \upharpoonright V_{\ell})
=-\sum_{j=1}^m\,\trace\bigl((\ad_{U_j}\restr{\mg_\ell})^2\bigr).
\end{equation}
Since $g_0$ is bi-invariant, $\mg_\ell$ is $g_0$-orthogonal to
$\mg_{\ell'}$ for $\ell\ne\ell'$.  Let $n_{\ell}$ denote the dimension of
$\mg_{\ell}$, and 
let $n=n_1+\ldots+n_r$ be the dimension of~$\mg_{ss}$. Choose a
$g_0$-orthonormal basis
$\{X_1,\ldots,X_n\}$ of~$\mg_{ss}$ such that the first $n_1$ elements lie
in~$\mg_1$, the next $n_2$ elements lie in $\mg_2$, etc.  Complete to a
$g_0$-orthonormal basis $\mathcal{B}_0=\{X_1,\dots,X_n,Z_1,\dots, Z_k\}$, where
(necessarily) $Z_1,\dots, Z_k\in\mt$.

Let $g$ be a metric in~$\mU$ which is eigenvalue equivalent to~$g_0$ up to level
$N$ and satisfies $\vol(g)\leq\vol(g_0)$. 
Since $g\in \mU'$ and $N\geq N'$, statement~(\ref{eq.isom}) holds and
thus there exist elements $W_i\in\mg_{ss}$ such that $\{Z_1+W_1,\dots,
Z_k+W_k\}$ is a $g$-orthonormal basis of $\mg_{ss}^{\perp_g}$.  Complete to a
$g$-orthonormal basis $\mathcal{B}=\{Y_1,\dots, Y_n,Z_1+W_1,\dots, Z_k+W_k\}$ of
$\mg$ with $Y_1,\dots, Y_k\in\mg_{ss}$.  The change of basis matrix which
expresses
the elements of~$\mathcal{B}$ in terms of~$\mathcal{B}_0$ is given by
$$
\begin{bmatrix}A&R\\0&I_k\end{bmatrix},
$$
where $A = (a_{ij}) \in \Mat_{n \times n} (\R)$, $R = (r_{ij}) \in \Mat_{n
\times k}(\R)$ and $I_k$ is the $k \times k$ identity matrix.
Hence, $Y_j=\sum_{i=1}^n a_{ij}X_i$ for $j=1,\ldots,n$ and $W_s=\sum_{i=1}^n
r_{is}X_i$ for $s=1,\ldots,k$.
The condition $\vol(g)\leq\vol(g_0)$ implies that
$|\det(A)|\ge1$.  Without loss of generality we assume $\det(A)>0$ and hence
$\det(A)\ge1$.

Since $g_0$ is bi-invariant, there exist numbers $c_\ell>0$ for
$\ell=1,\ldots,r$
such that the restriction of~$g_0$ to~$\mg_\ell$ coincides with $-c_\ell
B_\ell$, where $B_\ell$
is the Killing form of~$\mg_\ell$ (which in turn coincides with the restriction
to $\mg_\ell$
of the
Killing form $B:(X,Y)\mapsto\trace(\ad_X\circ\ad_Y)$ of $\mg$). In
particular, by equation~(\ref{eq.tr}), we have
$$
\trace(\Delta_0\upharpoonright V_\ell) = \frac{n_\ell}{c_\ell}.
$$
Since $g\in\mU\subset\mU''$ and $N\geq N''$, we have $\Delta_g\upharpoonright
V_\ell
= \Delta_0\upharpoonright V_\ell$ for each $\ell=1,\ldots,r$. In particular, for
$\ell=1$:
\begin{align*}
\frac{n_1}{c_1}  &= \trace(\Delta_0 \upharpoonright V_1)
       =  \trace(\Delta_g \upharpoonright V_1)
       = - \sum_{j=1}^n \trace \bigl((\ad_{Y_j} \upharpoonright \mg_1)^2\bigr)
        - \sum_{s=1}^k \trace \bigl((\ad_{Z_s+W_s} \upharpoonright
\mg_1)^2\bigr)\\
      &= - \sum_{i=1}^{n_1}\sum_{j=1}^n a_{ij}^2\trace
              \bigl((\ad_{X_i}\upharpoonright\mg_1)^2\bigr)
         - \sum_{i=1}^{n_1}\sum_{s=1}^k r_{is}^2\trace
               \bigl((\ad_{X_i}\upharpoonright\mg_1)^2\bigr)\\
      &= \frac1{c_1}\sum_{i=1}^{n_1}\sum_{j=1}^n a_{ij}^2
         +\frac1{c_1}\sum_{i=1}^{n_1}\sum_{s=1}^k r_{is}^2,
\end{align*}
where the fourth equality holds because $\ad_X\upharpoonright \mg_1=0$ for
$X\in\mg_\ell$ with $\ell\ne1$ and because $\{X_1,\ldots,X_{n_1}\}$ is
orthonormal with respect to $-c_1B_1$.
We hence obtain $n_1=\sum_{i=1}^{n_1}\sum_{j=1}^n a_{ij}^2
+\sum_{i=1}^{n_1}\sum_{s=1}^k r_{is}^2$. Summing over the analogous equations
for $\ell=1,\ldots,r$ we conclude that
$$
n=\|A\|^2+\|R\|^2,
$$
where $\|\,.\,\|$ denotes the standard Euclidean norm of matrices viewed as
points in
the appropriate~$\R^N$.
However, $n$ is the minimal value (in fact, the only critical value) of the
function
$\SL(n,\R)\ni C\mapsto\|C\|^2\in\R$ and is attained precisely on $\SO(n)$. It
thus follows
from $\det(A)\ge1$ that $A\in\SO(n)$ and $R=0$; hence $g=g_0$.
\end{proof}

\begin{rem}
Some of the techniques used in this section are similar to those used by Urakawa
in \cite{Urakawa}.
\end{rem}

\section{A stronger spectral isolation result for simple
groups}\label{Sec:Simple}

In the proof of Lemma~\ref{Lem:TraceNbhd} it was observed that if $g$ is a
left-invariant metric on a compact Lie group $G$, with associated Laplacian
$\Delta_g$, then any subspace $V \leq L^{2}(G)$ that is invariant under the
right-regular representation of $G$ is also invariant under $\Delta_g$. With
this in mind we have the following result concerning the trace of the Laplacian
on compact simple Lie groups.

\begin{prop}\label{prop:simple}
Let $g_0$ be a bi-invariant metric on a compact simple Lie group~$G$, and let
$g\ne g_0$ be a left-invariant metric on~$G$ which satisfies $\vol(g) \leq
\vol(g_0)$.  Then there exists a constant $C = C(g) >1$ such that 
$$\trace(\Delta_g
\upharpoonright V)=C\trace(\Delta_0 \upharpoonright V)$$
for every finite dimensional subspace $V\leq L^2(G)$ which is invariant
under the right-regular representation $\rho$
of~$G$ and on which $G$ acts nontrivially.

\end{prop}

\begin{proof}
By rescaling $g$ and $g_0$ we can assume without loss of generality
that $g_0$ coincides with $-B$ on the Lie algebra~$\mg$ of~$G$,
where $B$ is the Killing form on~$\mg$.
Define $h:\mg\times\mg\to\R$ by
$$
h(X,Y):=-\trace\bigl(((\rho_*X)\upharpoonright V)\circ((\rho_*Y)\upharpoonright
V)\bigr).
$$
Obviously $h$ is bilinear, symmetric, and $\Ad_G$-invariant. The map
$X\mapsto(\rho_*X)\upharpoonright V$ is
nonzero because $V$ is not a trivial representation space of~$G$.
Since $\mg$ is simple, this map has trivial kernel, which implies
that $h$ is positive-definite; in particular, there
exists some $c>0$ such that $h=-c B$.
We now proceed similarly as in the last part of the proof of
Theorem~\ref{Thm:Main}:
Let $\{X_1,\ldots,X_n\}$ be a $g_0$-orthonormal basis and $\{Y_1,\ldots,Y_n\}$
be a $g$-orthonormal basis of~$\mg$, and define $A=(a_{ij})$ by
$Y_j=\sum_{i=1}^n a_{ij}X_i$ for $j=1,\ldots,n$. By the volume condition,
$|\det(A)|\ge1$; we can assume $\det(A)\ge1$. Moreover, $g\ne g_0$ implies
$A\notin\SO(n)$ and therefore $\|A\|^2>n$. Thus,
\begin{align*}
\trace(\Delta_g\upharpoonright V)
&=\sum_{j=1}^n h(Y_j,Y_j)=-c\sum_{j=1}^n B(Y_j,Y_j)
=-c\sum_{i,j=1}^n a_{ij}^2 B(X_i,X_i) = c\|A\|^2\\
&>c n=\sum_{i=1}^n h(X_i,X_i)=\trace(\Delta_0\upharpoonright V).
\end{align*}
The proposition follows with $C=\frac{\|A\|^2}{n}$.
\end{proof}

\begin{rem}\label{rem:global}
Since volume is a spectral invariant the previous proposition implies the
following:
Let $g_0$ be a bi-invariant metric on a compact simple Lie group~$G$,
and suppose there exists a left-invariant metric~$g\ne g_0$ on~$G$
which is isospectral to~$g_0$.
Then $\trace(\Delta_g\upharpoonright V)>\trace(\Delta_0\upharpoonright V)$
for any finite dimensional invariant subspace~$V$ of~$L^2(G)$ on which
$G$ acts nontrivially.
Thus $\Delta_g$, although isospectral to~$\Delta_0$, must have greater
trace than~$\Delta_0$ on every $\Delta_0$-eigenspace (except for
the eigenvalue~$0$), even on each irreducible subspace.
This is not a priori a contradiction because some wild reordering
of eigenvalues could occur to produce this situation. Nevertheless,
this seems a strong indication in support of the conjecture that a bi-invariant
metric on a compact simple Lie group is globally spectrally determined
among left-invariant metrics.
\end{rem}

\begin{thm}\label{Thm:SimpleLocRigidity}
 Let $g_0$ be a bi-invariant metric on a compact simple Lie group~$G$.
Let $\alpha_1<\alpha_2<\alpha_3$ be three consecutive elements of the eigenvalue
set 
of $(G,g_0)$. {\rm(}See Definition~\ref{def.eig}.{\rm)}
Then there exists
a neighborhood $\mU$ of~$g_0$ in $\MIG$ {\rm(}depending on $\alpha_1$,
$\alpha_2$,
$\alpha_3${\rm)} such that if $g \in \mU$
satisfies $\vol(g) \leq \vol(g_0)$ and if $\alpha_1$, $ \alpha_2$ and $
\alpha_3$ are also consecutive elements of the eigenvalue set of $(G,g)$,
then $g = g_0$.
\end{thm}

\begin{cor}
A bi-invariant metric on a compact simple Lie group is
locally determined within the set of left-invariant metrics of
at most the same volume
by its first two distinct non-zero eigenvalues $ 0 < \lambda_1 < \lambda_2$
{\rm(}ignoring multiplicities{\rm)}.
\end{cor}

The corollary follows from the theorem since $0,\lambda_1,\lambda_2$ are three
consecutive elements of the eigenvalue set.

\begin{proof}[Proof of Theorem~\ref{Thm:SimpleLocRigidity}]
Since $\Delta_0$ commutes with right translations in~$G$, the
$\alpha_2$-eigenspace~$V$
of~$\Delta_0$ is invariant under the right-regular representation.
Note that $V$ is finite dimensional.
As remarked in the proof of Lemma~\ref{Lem:TraceNbhd}, the eigenvalues
of $\Delta_g \upharpoonright V$ depend continuously on~$g$.
Thus there exists a neighborhood~$\mU$
of~$g_0$ in $\MIG$ such that for any $g \in \mU$ the eigenvalues of $\Delta_g$
on~$V$ are contained in the interval $(\alpha_1, \alpha_3)$.
Let $g \in\mU$ be a metric which satisfies $\vol(g)\leq\vol(g_0)$ and the
condition
that $\alpha_1, \alpha_2$ and $\alpha_3$ are also consecutive eigenvalues of
$\Delta_g$.
Then necessarily $\Delta_g\upharpoonright
V=\alpha_2\Id_V=\Delta_0\upharpoonright V$.
Finally, note that $G$ acts nontrivially on~$V$ since $\alpha_2\ne0$.
Proposition~\ref{prop:simple} now implies $g=g_0$.
\end{proof}


\bibliographystyle{amsalpha}


\end{document}